\documentclass[12pt]{article}

\usepackage{amsmath, amssymb,amscd,amsthm,amsfonts}
\usepackage{bm, mathrsfs,dsfont}
\usepackage{float}  
\usepackage{tikz}
\usepackage{bbding}  
\newtheorem{theorem}{Theorem}[section]

\newtheorem{defn}{Definition}[section]
\newtheorem{lemma}{Lemma}[section]

\newtheorem{exam}{Example}[section]
\newtheorem{rem}{Remark}[section]
\def\nd{\noindent}
\def\fX{\frak X}
\def\fD{\frak D}
\def\fP{\frak P}

\title{\bf Bounded Projective Functions and Hyperbolic Metrics with Isolated Singularities}
\author{Bo Li$^1$, Yu Feng$^2$, Long Li$^3$ and Bin Xu$^4$}

\begin{document}
\maketitle

\noindent{\small $^{1,2,4}$Wu Wen-Tsun Key Laboratory of Math, USTC, Chinese Academy of Science\\
School of Mathematical Sciences, University of Science and Technology of China.\\
No. 96 Jinzhai Road, Hefei, Anhui Province 230026 P. R. China. \\
$^3$ Fourier Institute, 100 Rue des Maths 38610 Gi\' eres, Grenoble, France.\\

\noindent{\small $^{1}$ ilozyb@mail.ustc.edu.cn \quad $^{2}$ fengyu0954@163.com\\
\quad $^3$Long.Li1@univ-grenoble-alpes.fr \quad $^4$ \Envelope bxu@ustc.edu.cn}
\vskip0.5cm

\paragraph{Abstract}
{\small
we establish a correspondence on a Riemann surface between hyperbolic metrics with isolated singularities  and bounded projective functions whose Schwarzian derivatives have at most double poles and whose monodromies lie in ${\rm PSU}(1,\,1)$. As an application, we construct explicitly a new class of hyperbolic metrics with countably many singularities on the unit disc.}

\section{Introduction}
\paragraph{}
Let $X$ be a compact Riemann surface
and $D=\sum_{j=1}^n\,(\theta_j-1)\,P_j$ be a ${\Bbb R}$-divisor on $X$ such that $1\not=\theta_j\geq 0$ and $P_1,\cdots,P_n$ are $n$ distinct points on $X$.
We call $ds^2$ a {\it conformal metric representing $D$} if $ds^2$ is a smooth conformal metric on $X\backslash {\rm supp}\, D:=X\backslash \{P_1,\cdots, P_n\}$ and
in a neighborhood $U_j$ of $P_j$, $ds^2$ has form $e^{2u_j}\,|dz|^2$, where $z$
is a local complex coordinate defined in $U_j$
centered at $P_j$ and the smooth real valued function
\begin{equation*}
v_j:=
\begin{cases}
u_j-(\theta_j-1)\,\ln\,|z|\quad & \quad {\rm if}\ \theta_j>0\\

u_j+\ln\,|z|+\ln\,\big(-\ln\,|z|\big)\quad &\quad {\rm if}\ \theta_j=0
\end{cases}
\end{equation*}
on $U_j\backslash \{P_j\}$ extends to a continuous function on $U_j$.
We also say that $ds^2$ has a {\it cone singularity of angle $2\pi\theta_j$ at $P_j$} when $\theta_j>0$,
and it has a {\it cusp singularity at $P_j$} when $\theta_j=0$.
We may think of the above definition of cone/cusp singularity a differential geometric one.
There also exists also a unified complex analytical definition for both cone singularity and cusp one of a hyperbolic metric (Definition \ref{defn:cx}). We prove in Lemma \ref{lem:equiv} that these two definitions coincide with each other when we are considering hyperbolic metrics.
J. Nitsche \cite{Nits57} and M. Heins \cite[$\S$\,18]{Hei62} proved that {\it an isolated singularity of a  hyperbolic metric must be either a cone singularity or a cusp one}.
By the Gauss-Bonnet formula, if $ds^2$ is a conformal hyperbolic metric representing $D=\sum_{j=1}^n\,(\theta_j-1)\,P_j$ on $X$, then there holds
$$\chi(X)+\sum_{j=1}^n\,(\theta_j-1)<0,$$ where $\chi(X)$ is the Euler number of $X$.
M. Heins \cite{Hei62} studied the properties of S-K metrics and applied it to showing

\begin{theorem} {\rm (\cite[Chapter II]{Hei62})}
There exists a unique conformal hyperbolic metric representing an ${\Bbb R}$-divisor $D=\sum_{j=1}^n\,(\theta_j-1)\,P_j$ with $\theta_j\geq 0$ on a compact Riemann surface $X$ if and only if
$\chi(X)+\sum_{j=1}^n\,(\theta_j-1)<0$.
\end{theorem}

\nd Actually the history of this problem goes back to Picard \cite{Pic1905}, who studied the hyperbolic metrics with cone singularities. After Heins' work \cite{Hei62}, both McOwen \cite{MO88} and Troyanov \cite{Tr91}, who were unaware of M. Heins' work \cite{Hei62} apparently,  proved the theorem for the case of $\theta_j>0$ by using different PDE methods. Moreover, the hyperbolic metrics on the Riemann sphere with three singularities were expressed explicitly in \cite{ASVV2010,KRS2011, Zh2014} by using the Gauss hypergeometric  functions and some refined properties of the metrics were also studied there.

In this manuscript, by using developing map \cite[$\S\,$3.4]{Thur97} and \cite[Sections 2-3]{CWWX15} and Complex Analysis,  we shall investigate more general hyperbolic metrics with isolated singularities on Riemann surfaces, which are {\it not} necessarily compact.
To this end, we need to prepare some notions at first. Let ${\frak X}$ be a Riemann surface and ${\frak D}=\sum_j\,(\theta_j-1)\,{\frak P}_j$ an ${\Bbb R}$-divisor on ${\frak X}$ such that
$\theta_j\geq 0$ and ${\frak P}_1,\,{\frak P}_2,\cdots$ are mutually distinct points on $\fX$, which form a closed and discrete subset of ${\frak X}$. We denote $\{{\frak P}_1,\,{\frak P}_2,\cdots\}$ by ${\rm supp}\, {\frak D}$, which is at most countable. We could give the definition  of {\it conformal metric representing ${\frak D}$ on ${\frak X}$} in the similar way as the first paragraph of this section.
 We call a multi-valued locally univalent meromorphic function
$f:S\to {\Bbb P}^1:={\Bbb C}\cup \{ \infty \} $ on a Riemann surface $S$ a {\it projective function} if the monodromy of $f$ lies in the group ${\rm PSL}(2,\,{\Bbb C})$ consisting of
all M{\" o}bius transformations.
We call a projective function $f$ on $\fX\backslash {\rm supp}\, \fD$ {\it compatible with} $\fD$ if and only if
the Schwarzian derivative $\{f,\,z\}=\Bigl(\frac{f''(z)}{f'(z)}\Bigr)' - \frac{1}{2}\Bigl(\frac{f''(z)}{f'(z)}\Bigr)^2$ of $f$ has the form of  $\frac{1-\theta_j^2}{2z^2}+\frac{b_j}{z}+h_j(z)$ near each $\fP_j$, where $z$ is a complex coordinate centered at $\fP_j$, $b_j$ is a constant, and $h_j(z)$ is holomorphic near $\fP_j$. We note that both the constant $b_j$ and the holomorphic function $h_j(z)$ depend on the choice of the complex coordinate $z$, but the principal singular term $\frac{1-\theta_j^2}{2z^2}$  of $\{f,\,z\}$ does not.

\begin{theorem}
\label{thm:conn}
There exists a conformal hyperbolic metric $ds^2$ representing an ${\Bbb R}$-divisor $\fD=\sum_{j}\,(\theta_j-1)\,\fP_j$ with $\theta_j\geq 0$ on a Riemann surface $\fX$ if and only if there exists a projective function
$f:\fX\backslash {\rm supp}\, \fD\to {\Bbb D}:=\{w\in {\Bbb C}:\,|w|<1\}$ such that $f$ is compatible with $\fD$ and
the monodromy of $f$ lies in the holomorphic automorphism group
$${\rm PSU}(1,\,1)=\left\{w\longmapsto\frac{aw+b}{\overline{b}w+\overline{a}}:\ a,\ b\in \mathbb{C},\ \vert a \vert^2-\vert b \vert^2=1\right\}$$
 of ${\Bbb D}$.
Moreover, $ds^2$ coincides with the pullback $f^* g_{\rm st}$ of the standard hyperbolic metric $g_{\rm st}:=\frac{4|dw|^2}{(1-|w|^2)^2}$ on ${\Bbb D}$ by $f$. We call $f$ a {\rm developing map} of the metric $ds^2$, which is uniquely determined up to a post-composition with an automorphism of ${\Bbb D}$.
\end{theorem}

\begin{rem}
{\rm Feng Luo \cite{Luo93} also mentioned the same result as the theorem for hyperbolic metrics with only cone singularities on compact Riemann surfaces. The statement of the theorem near a {\it cusp} singularity is new and non-trivial. Actually, it is reduced to showing that the differential geometric definition coincides with the complex analytical one for cusp singularity of a hyperbolic metric.
The proof of the coincidence is interesting because we need to do some subtle analysis for a second order nonlinear elliptic partial differential equation.   }

\end{rem}

\begin{rem}
{\rm Q. Chen, W. Wang, Y. Wu and the last author \cite[Theorem 3.4]{CWWX15} proved an analogue of Theorem \ref{thm:conn} for cone spherical (constant curvature one) metrics on compact Riemann surfaces, which motivated us to come up with Theorem \ref{thm:conn}. For the convenience of the readers, we state a more general version of \cite[Theorem 3.4]{CWWX15}  as follows:
{\it There exists a conformal cone spherical metric $\widetilde{ds^2}$ representing a ${\Bbb R}$-divisor $\fD=\sum_{j}\,(\theta_j-1)\,\fP_j$ with $\theta_j>0$ on a Riemann surface $\fX$ if and only if there exists a projective function
$\tilde{f}:\fX\backslash {\rm supp}\, \fD\to {\Bbb P}^1$ such that $\tilde{f}$ is compatible with $\fD$ and
the monodromy of $\tilde{f}$ lies in  the group
$${\rm PSU(2)}=\left\{w\longmapsto\frac{aw+b}{-\overline{b}w+\overline{a}}:\ a,\ b\in \mathbb{C},\ \vert a \vert^2+\vert b \vert^2=1\right\}$$
consisting of all the M{\" o}bius transformations preserving
the standard spherical metric $\frac{4|dw|^2}{(1+|w|^2)^2}$ on the Riemann sphere ${\Bbb P}^1$ . Moreover,  $\widetilde{ds^2}=\tilde{f}^*\Big(\frac{4|dw|^2}{(1+|w|^2)^2}\Bigr)$.}
Of course,  Theorem \ref{thm:conn} and its proof are  different from the spherical analogue of them in the sense cusp singularities do not at all appear in spherical metrics.
}
\end{rem}

As an application of Theorem \ref{thm:conn}, we find the following new example of hyperbolic metrics with countably many singularities.

\begin{exam}
\label{exam:disc}
Let $\sum_{j=1}^\infty a_n$ be a convergent series of positive numbers and $\{z_j\}_{j=1}^\infty$ a closed discrete subset of ${\Bbb D}$. Then $h(z):=\sum_{j=1}^\infty\,\frac{a_j}{z-z_j}$
is a meromorphic function on ${\Bbb D}$ and there exists a real number $\lambda_0$ and a one-parameter  family $\{ds^2_\lambda:\,\lambda>\lambda_0\}$ of conformal hyperbolic metrics representing the same ${\Bbb Z}$-divisor ${\frak D}=(h)$ on ${\Bbb D}$. That is, these metrics have cusp singularities at $z_j$'s and a cone singularity of angle $2\pi\big(1+{\rm ord}_w(h)\big)$ at each zero $w$ of $h$.

We obtain a similar statement if we use the finite sum $\sum_{j=1}^N\, a_n$ of positive numbers and a finite subset $\{z_j\}_{j=1}^N$ of ${\Bbb D}$, where $\sum_{j=1}^N\,\frac{a_j}{z-z_j}$ has $(N-1)$ zeros
{\rm (}counting multiplicities{\rm )} on ${\Bbb D}$. We don't know whether $h(z):=\sum_{j=1}^\infty\,\frac{a_j}{z-z_j}$ vanishes on ${\Bbb D}$  although we observe that $h$ has no zero on $\{z\in {\Bbb C}:\, |z|\geq 1\}$. As a compensation, we show that
$$h_0(z)=\sum_{j=1}^\infty\,\frac{1}{2j^3(2j+1)}\,\cdot\, \frac{1}{z-\Big(1-\frac{1}{2j-1}\Big)}$$
has infinitely many zeros on ${\Bbb D}$.

  \end{exam}

\begin{rem}{\rm
Using the punctured disc case of Theorem \ref{thm:conn}, Yiqian Shi and the second and the last authors \cite{Feng2017} obtained an explicit local model of an isolated singularity of a hyperbolic metric
in a suitably chosen complex coordinate around the singularity. Moreover, based on the correspondence in this theorem and following the ideas in \cite{CWWX15,SCLX18}, they \cite{Feng2019} have been investigating systematically new hyperbolic metrics with isolated singularities on noncompact Riemann surfaces by using both subharmonic functions and Abelian differentials. }
\end{rem}

We conclude this section by explaining the organization of this manuscript. In Section 2, we give the complex analytical definition (Definition \ref{defn:cx}) of both cone and cusp singularities of a hyperbolic metric and prove that the differential geometric definition implies the complex analytical one by the PDE method. As a consequence, we obtain the necessary part of Theorem \ref{thm:conn}.
In the last section, we prove that complex analytical definition implies that the differential geometric one, which also implies the sufficient part of Theorem \ref{thm:conn}.
We also provide in this section the details of Example \ref{exam:disc} and propose three questions.

\section{The complex analytical definition}
\paragraph{}

At first we give the complex analytical definition of cone/cusp singularity of a hyperbolic metric.
Consider a hyperbolic metric $ds^2=e^{2u}|dz|^2$ on the punctured disc $U^*:=U-\{0\}=\{0<|z|<1/2\}$, where $U=\{|z|<1/2\}$.
By the similar argument as \cite[Lemma 2.1]{CWWX15}, there exists a projective function $f:U^*\to {\Bbb D}$ with monodromy in ${\rm PSU(1,\,1)}$ such that $ds^2=f^*g_{\rm st}$.
We call $f$ a {\it developing map} of the hyperbolic metric $ds^2$, which is unique up to a post-composition by a M\" obius transformation in ${\rm PSU(1,\,1)}$.
Moreover, the Schwarzian derivative $\{f,\,z\}$ of $f$ is meromorphic in $U^*$.

\begin{defn}
\label{defn:cx}
{\rm We call that  $z=0$ is a {\it cusp singularity} of a hyperbolic metric $ds^2$ on $U^*$ if and only if near $z=0$ there holds
$\{f,\,z\}=\frac{1}{2z^2}+\frac{b_0}{z}+h(z)$, where $b_0$ is a constant and $h(z)$ is holomorphic near $z=0$.
We call that  $z=0$ is a {\it cone singularity} with angle $2\pi\theta>0$ of $ds^2$ if and only if near $z=0$ there holds
$\{f,\,z\}=\frac{1-\theta^2}{2z^2}+\frac{c_0}{z}+g(z)$, where $c_0$ is a constant and $g(z)$ is holomorphic near $z=0$.}
\end{defn}

\begin{lemma}
\label{lem:equiv}
The differential geometric definition of  cone/cusp singularity of a hyperbolic metric in $U^*$ in the first paragraph of Section 1 coincides with the complex analytical one as in Definition \ref{defn:cx}.
\end{lemma}

We leave to the next section the proof that the complex analytical definition implies the differential geometric one. We shall prove the implication of the opposite direction in what follows.

Suppose that in the differential geometric sense, $z=0$ is a cone singularity of angle $2\pi\theta>0$ of the hyperbolic metric $ds^2=e^{2u}|dz|^2$ in $U^*=\{0<|z|<1/2\}$. Then, by the similar computation in \cite[Lemma 3.1]{CWWX15} and \cite[Lemma, Section 3]{Tr89}, we find that the Schwarzian derivative of $f$ equals $2\Big(\frac{\partial ^2u}{\partial z^2}-(\frac{\partial u}{\partial z})^2 \Big)$, which has the form of
$\frac{1-\theta^2}{2z^2}+\frac{c_0}{z}+g(z)$, where $c_0$ is a constant, and $g(z)$ is holomorphic near $z=0$.
Hence, $z=0$ is also a cone singularity in the sense of Definition \ref{defn:cx}.
However, the same argument could not go through for $\{f,\,z\}$ if $z=0$ is a cusp singularity of $ds^2$ in
the differential geometric sense.
In the remaining part of this section, we shall show by a different PDE method from \cite[Lemma, Section 3]{Tr89} that {\it $f$ is a cusp singularity in the sense of Definition \ref{defn:cx}}.

Suppose that $z=0$ is a cusp singularity of the hyperbolic metric $ds^2=e^{2u}|dz|^2$ in $U^*=\{0<|z|<1/2\}$ in the differential geometric sense. By the very definition, $v:=u+\ln\,|z|+\ln(-\ln\, |z|)$ is continuous on $U$. Then we have
$$ds^2=e^{2u}|dz|^2=f^* g_{\rm st}=\frac{4|f'|^2|dz|^2}{(1-|f|^2)^2}\quad {\rm and}\quad  u=\ln\, 2+\ln\,|f'|-\ln\, (1-|f|^2).$$
It suffices to show that
$\{f,\,z\}=\frac{1}{2z^2}+\frac{d}{z}+\psi(z)$,
where $d$ is a constant and $\psi$ is holomorphic in $U$.
By computation, there holds in $U^*$ that
\begin{eqnarray*}
\{f,\,z\}=2\Bigg(\frac{\partial ^2u}{\partial z^2}-\bigg(\frac{\partial u}{\partial z}\bigg)^2 \Bigg)
=\dfrac{1}{2z^2}+\dfrac{2}{z} \Biggl( \dfrac{\partial v}{\partial z}\bigg(1+\frac{1}{\ln\,|z|}\bigg)+z\dfrac{\partial^2 v}{\partial z^2}-z\bigg(\dfrac{\partial v}{\partial z}\bigg)^2\Biggr),
\end{eqnarray*}
which is holomorphic in $U^*$ since $ds^2$ is hyperbolic there \cite[Lemma, Section 3]{Tr89}.
The problem is reduced to showing

\begin{lemma}
 {\it The holomorphic function
$$F(z):=\dfrac{\partial v}{\partial z}\bigg(1+\frac{1}{ln|z|}\bigg)+z\dfrac{\partial^2 v}{\partial z^2}-z\bigg(\dfrac{\partial v}{\partial z}\bigg)^2$$ in $U^*$ extends to $z=0$.}
\end{lemma}

We need two lemmas for the proof of Lemma 2.2.

\begin{lemma}
Denote ${\bf D}:=\{|z|<1/5\}\subset U$. Then
$\int_{{\bf D}} |\nabla v|^2 <+\infty. $ Here and later on we omit in the integrals the standard Lebesgue measure $\frac{\sqrt{-1}}{2}\, {\rm d}z\wedge {\rm d}{\bar z}$ on the Euclidean plane ${\Bbb C}\supset {\bf  D}$.
\end{lemma}

\begin{proof} The proof is divided into three steps.

{\it Step 1} Since $g=e^{2u}\,|dz|^2$ is a hyperbolic metric on $U^*$, we have $\Delta u=4\,\frac{\partial^2 u}{\partial z \partial {\bar z}}=e^{2u}$ on $U^*$. Since $v=u+\ln\,|z|+\ln(-\ln\, |z|)$, we rewrite the former equation as the following one
\[\Delta  v=h:=\dfrac{e^{2v}-1}{|z|^2(\ln\,|z|)^2}\quad {\rm in}\quad U^*.\]
Recall that $v$ is continuous in $U$ and $h$ is locally integrable in $U$.  Hence, both sides of this equation can be thought of as distributions in $U$.
Now we shall prove that {\it this equation holds in $U$ in the sense of distribution.}

As a distribution, the support of $\Delta v-h $ is contained in $\{z=0\}$. By \cite[Theorem 2.3.4]{Hor90}, $\Delta v-h$ equals a linear combination of the Dirac delta function $\delta_0$ and its partial derivatives, i.e. $ \Delta v-h=\Sigma C_\alpha \partial^\alpha \delta_0 $, where there are at most finitely many nonzero constants $C_\alpha$. Take an arbitrary multi-index $\alpha$ and fix it. We can choose a function
$\phi\in C_0^\infty(U)$ such that $ \partial^\alpha \phi(0) \ne 0 $ and  $ \partial^\beta \phi(0) = 0 $ for all $ \beta \ne \alpha $. Denote $\phi_k(z)=\phi(kz)$. Then
$$ h(\phi_k)=\int_{U} h\phi_k \le \sup|\phi_k|\int_{{\rm supp}\,\phi_k} h\to 0\quad {\rm as}\quad k\to \infty. $$
Moreover, since $ \int_U \Delta \phi=0  $ and $v$ is continuous at $z=0$, we have
$$ | \Delta v(\phi_k)| =\left| \int_U v(z/k) \Delta \phi -\int_{U} v(0) \Delta \phi \right| \le \int_{U} |v(z/k)-v(0)| |\Delta \phi| \to 0.$$
Hence $ (\Delta v -h)(\phi_k)\to 0 $ as $k\to \infty$. On the other hand, $ (\Delta v -h)(\phi_k) = C_\alpha k^{|\alpha|}  \partial^\alpha \phi(0)$, which implies $ C_\alpha=0 $. Thus  $ \Delta v=h $ on $U$ as distributions.

{\it Step 2} Recall that $v$ is smooth in $U^*$ and continuous at $z=0$.
 We shall prove $$ \int_{\bf{D}^*} |\nabla v|^2 <+\infty .$$
Choose a family  $\{\chi_\epsilon:\,\epsilon>0\}$ of compactly supported non-negative smooth functions in ${\Bbb C}$ such that $ \int_{\Bbb C} \chi_\epsilon =1 $ and ${\rm supp}\, \chi_\epsilon \subset \{|z|\leq \epsilon\}$. Since $ v $ is continuous on $U$,  the convolutions $ v_k:=\chi_\frac{1}{k} * v $, $k=5,6,7,\cdots$, are well defined
smooth functions in $\overline{\bf D}$, which converge uniformly to $v$ on $\overline{\bf D}$ as $k\to\infty$.  Moreover, since $ v\in C^\infty(\bf{D}^*) $, as $k\to\infty$, $ |\nabla v_k|^2 \to |\nabla v|^2 $ uniformly in any compact subsets of $ \bf{D}^* $. By Fatou's lemma, we have
 $$  \int_{\bf{D}^*} |\nabla v|^2 \le  \liminf\limits_{k} \int_{\mathbb{D}^*} |\nabla v_k|^2. $$
Then we show that {\it the integrals $ \int_{\bf{D}^*} |\nabla v_k|^2 $ are uniformly bounded for all $k=5,6,7,\cdots$}. Using integration by part, we have
 $$  \int_{\mathbb{D}^*} |\nabla v_k|^2 =\int_{\mathbb{D}} |\nabla v_k|^2 =-\int_{\mathbb{D}} v_k\Delta v_k + \int_{\partial \mathbb{D}} v_k \frac{\partial v_k}{\partial \vec{n}}. $$
Recall that $v_k\to v$ uniformly on $\overline{\bf D}$ and $  \frac{\partial v_k}{\partial \vec{n}} \to \frac{\partial v}{\partial \vec{n}}$ uniformly on $ \partial \bf{D} $.
 The Problem is reduced to showing that $ \int_{\bf{D}} |\Delta v_k|$ is uniformly bounded. Actually, as $k\geq 5$, we have
 \begin{eqnarray*}
 \int_{\bf{D}} |\Delta v_k| &=& \int_{|z|<1/5}\left|\int_{|\tilde{z}|<2/5} \chi_{1/k}(z-\tilde{z})\Delta v(\tilde{z})\right| \\
 &\le & \int_{\Bbb C} \chi_{1/k} \int_{|z|<2/5} |\Delta v|=\int_{|z|<2/5}\,|h| <\infty.
 \end{eqnarray*}
 Thus we conclude $ \int_{\bf{D}^*} |\nabla v|^2 <+\infty $.

{\it Step 3} Denote the standard coordinate $z$ in ${\bf D}\subset {\Bbb C}$ by $z=x+\sqrt{-1}y$.
Then $w:=\frac{\partial v}{\partial x}$ is a smooth and square integrable function on $ \bf{D}^* $, which
can be thought of as an square integrable function and then a distribution in ${\bf D}$. The partial derivative $\frac{\partial v}{\partial x}$ of the continuous function $v$ in ${\bf D}$ is also a distribution in ${\bf D}$. We shall show that {\it  the two distributions $\frac{\partial v}{\partial x}$ and $w$ coincide.}
Take a smooth test function $ \phi $ supported in $ \bf{D} $. By the Fubini theorem, we have
\begin{eqnarray*}
\bigg(\frac{\partial v}{\partial x}-w\bigg)(\phi)&=&\iint\,\bigg(-v\frac{\partial \phi}{\partial x}-w\phi\bigg)\\
 &=&\lim\limits_{\epsilon\to 0^+} \int_{-\infty}^\infty\, dy\int_{|x|>\epsilon} \bigg(-v\frac{\partial \phi}{\partial x}-w\phi\bigg) dx \\
&=&  \lim\limits_{\epsilon\to 0^+} \int_{-\infty}^\infty\, \Big( v(\epsilon,y)\phi(\epsilon,y)-v(-\epsilon,y)\phi(-\epsilon,y) \Big)dy\\&=&0.
 \end{eqnarray*}
The similar statement holds for $ \frac{\partial v}{\partial y} $. Therefore, we complete the proof.
\end{proof}

\begin{lemma}

$F(z)$ is in $L^{2-\epsilon} (\bf{D}) $ for all $0<\epsilon <1$. In particular, $z=0$ is at most a simple pole of $F(z)$.

\end{lemma}

\begin{proof} By Lemma 2.1, both $\frac{\partial v}{\partial z}$ and $\frac{\partial v}{\partial \bar{z}}$ belong to $L^2(\bf{D})$. Then the first summand $ \frac{\partial v}{\partial z}\big(1+\frac{1}{ln|z|}\big) $ in $F(z)$ also lies in $ L^2(\bf{D}) $.
Defining $ \tilde{v}:=zv $, we have
$$ \Delta \tilde{v} = 4\frac{\partial v}{\partial \bar{z}}+z\Delta v.$$
 Since both $\frac{\partial v}{\partial \bar{z}}$ and $z\Delta v=\dfrac{e^{2v}-1}{|z|(ln|z|)^2}$ belong to $L^2(\bf{D}) $, we have $ \Delta \tilde{v} \in L^2(\bf{D})$ and then $ \tilde{v} \in W^{2,2}(\bf{D})$.
 By the Sobolev embedding theorem, we obtain $ \frac{\partial \tilde{v}}{\partial z}=v+z\frac{\partial v}{\partial z} \in L^p(\bf{D})$ for any $p>1$.  Since $v\in L^p(\bf{D})$ as well ,we have $ z\frac{\partial v}{\partial z} \in L^p(\bf{D}) $ for any $p>1$. We now claim that {\it the third summand
  $ z \big(\frac{\partial v}{\partial z}\big)^2 $ in  $F(z)$ belongs to $ L^{2-\epsilon}(\bf{D}) $ for all $0<\epsilon<1$.} In fact, defining $\frac{1}{p}:=\frac{\epsilon}{2}$ and $ \frac{1}{q}:=\frac{2-\epsilon}{2} $, by the H{\" o}lder inequality, we obtain
 $$ \int_{\bf{D}} \left|z\bigg(\frac{\partial v}{\partial z}\bigg)^2\right|^{2-\epsilon}\leq \Bigg(\int_{\bf{D}} \left|z\dfrac{\partial v}{\partial z}\right|^{(2-\epsilon) p}\Bigg)^{1/p} \Bigg( \int_{\bf{D}} \left|\dfrac{\partial v}{\partial z}\right|^{(2-\epsilon)q} \Bigg)^{1/q} < \infty
 $$
As long as the the second summand $z\frac{\partial^2 v}{\partial z^2}$ in $F(z)$ is concerned, since $\frac{\partial ^2 \tilde{v}}{\partial z^2} = 2\frac{\partial v}{\partial z} + z\frac{\partial^2 v}{\partial z^2} \in L^2(\bf{D})$, we have $z\frac{\partial^2 v}{\partial z^2} \in L^2(\bf{D})$.
Therefore, we have proved  $ F(z)\in L^{2-\epsilon}(\bf{D}) $.

Since $F$ is holomorphic and integrable in ${\bf D}^*$,  $z=0$ is at most a simple pole of $F$.
\end{proof}

\noindent {\bf Proof of Lemma 2.2}\quad
We prove by contradiction.
Suppose that $F(z)$ has a simple pole at $z=0$ with residue $-\lambda ^2/4$, where $\lambda \in \mathbb{C}^*$ and $\Re \lambda \ge0$. Take a developing map $f:U^*\to {\Bbb D}$ of the restriction of the hyperbolic metric $ds^2$ to $U^*$.
Moreover,  we have $ \{f,z\}=\frac{1}{2z^2}-\frac{2F(z)}{z}=\frac{1-\lambda ^2}{2z^2}+\frac{d}{z}+\psi$ for some constant $d$ and a holomorphic function $\psi$ in $U$.
Near each small disk lying in $U^*:=U\backslash \{0\}$,   $f$ is the ratio of two linear independent solutions of the following Fuchsian equation
  $ \frac{d^2 y}{dz^2}+\frac{1}{2} \big(\frac{1-\lambda ^2}{2z^2}+\frac{d}{z}+\psi\big)y=0, $
whose two indicial exponents are $ (1+\lambda)/2 $ and $ (1-\lambda)/2 $ with difference $\lambda$.

If $\lambda\notin {\Bbb Z}$, it follows from the Frobenius method \cite[p.39]{Yo87} that there exists a small neighborhood, say $V$, of $z=0$, and another complex coordinate $\xi$ of $V$ centered at $z=0$ such that $f$ has the form of $\xi^\lambda$ in each small disk of $V^*:=V\backslash \{0\}$, where the details of computation is the same with \cite{Feng2017}. Since $f$ takes values in ${\Bbb D}^*$,  $\lambda$ must be positive and  $z=0$ is a cone singularity of $ds^2$ with angle $2\pi \lambda$. Contradiction!

If $\lambda\in {\Bbb Z}_{\not=0}$, then by a combination of the Frobenius method and the fact that $f$ takes values in ${\Bbb D}^*$ and has monodromy in ${\rm PSU(1,\,1)}$, we find that $\lambda$ is a positive integer and
$f$ has form $\xi^\lambda$ in another complex coordinate chart $(V,\,\xi)$ centered at $z=0$. It implies that $z=0$ is also a cone singularity of $ds^2$ of angle $2\pi\lambda$. Contradiction! $\hfill{\Box}$\\

At last we prove\\

\nd {\sc The necessary part of Theorem \ref{thm:conn}}\quad
By the similar argument as \cite[Lemma 2.1]{CWWX15}, there exists a projective function $f:
\fX\backslash {\rm supp}\,\fD\to {\Bbb D}$ with monodromy in ${\rm PSU(1,\,1)}$ such that $ds^2=f^*g_{\rm st}$.
We call $f$ a {\it developing map} of the hyperbolic metric $ds^2$, which is unique up to
a post-composition of a M\" obius transformation in ${\rm PSU(1,\,1)}$.
It follows from the proven part of Lemma \ref{lem:equiv} that $f$ is compatible with $\fD$.

\section{Sufficient part of Theorem 1.2, an example and three questions}
\paragraph{}

To complete the proof of Lemma \ref{lem:equiv}, we need only to show  that Definition \ref{defn:cx} for cone and cusp singularities of a hyperbolic metric $ds^2$ in $U^*=\{0<|z|<1/2\}$ implies
the differential geometric definition of them.
Actually, the argument is similar as in the proof of Lemma 2.2.
If $\theta\notin {\Bbb Z}$, then
by only using the Frobenius method, we find easily
that $z=0$ is a cone singularity of angle $2\pi\theta$ as $\theta>0$, and it is a cusp singularity as $\theta=0$. If $\theta\in {\Bbb Z}_{>1}$, since $|f|<1$, we could rule out the possibility that $f$ may have the logarithmic singularity at $z=0$ and find that $z=0$ is a cone singularity of angle $2\pi\theta_j$. \\

Then we prove\\

\nd{\sc The sufficient part of Theorem \ref{thm:conn}}\quad Suppose that
$f:\fX\backslash {\rm supp}\, \fD\to {\Bbb D}$ is a projective function which is compatible with $\fD$ and has
the monodromy in
${\rm PSU}(1,\,1)$. Then $f^*g_{\rm st}$ is a conformal hyperbolic metric on $\fX\backslash {\rm supp}\, \fD$.
It follows from Lemma \ref{lem:equiv} that the metric $f^*g_{\rm st}$ represents $\fD$.\\

Before giving the details of Example \ref{exam:disc}, we need an equivalent version of Theorem \ref{thm:conn} as follows. \\

\nd {\it There exists a conformal hyperbolic metric $ds^2$ representing an ${\Bbb R}$-divisor $\fD=\sum_{j}\,(\theta_j-1)\,\fP_j$ with $\theta_j\geq 0$ on a Riemann surface $\fX$ if and only if there exists a projective function
$f:\fX\backslash {\rm supp}\, \fD\to {\Bbb H}:=\{w\in {\Bbb C}:\,\Im w>0\}$ such that $f$ is compatible with $\fD$ and
the monodromy of $f$ lies in the holomorphic automorphism group
$${\rm PSL}(2,\,{\Bbb R})=\left\{w\longmapsto\frac{aw+b}{cw+d}:\ a,\ b,\ c,\ d\in \mathbb{R},\ ad-bc=1\right\}$$
 of ${\Bbb H}$.
Moreover, $ds^2$ coincides with the pullback $f^* \widetilde{g_{\rm st}}$ of the standard hyperbolic metric $\widetilde{g_{\rm st}}:=\frac{4|dw|^2}{\big(\Im w\big)^2}$ on ${\Bbb H}$ by $f$. We call $f$ a {\rm developing map} of the metric $ds^2$, which is uniquely determined up to a post-composition with an automorphism of ${\Bbb H}$.}\\

\nd Denote $\omega:=-\sqrt{-1}\,h(z)\, dz=\sum_{j=1}^\infty\,\Big(\frac{-\sqrt{-1}\,a_j}{z-z_j}\Big)\, dz.$
Since $\sum_{n=1}^\infty\, a_n$ is a convergent series of positive numbers,
we observe that  the multi-valued function $\int_0^z\, \omega$ on ${\Bbb D}\backslash\{z_j\}_{j=1}^\infty$
has monodromy in $\{w\mapsto w+t:\, t\in {\Bbb R}\}\subset{\rm PSL}(2,\,{\Bbb R})$ such that its imaginary part  $\Im\bigg(\int_0^z \omega\bigg)$ is single-valued and has a lower bound. Hence,
 there exists a real number $\lambda_0$ such that for all $z\in {\Bbb D}\backslash\{z_j\}_{j=1}^\infty$
\[\lambda_0+\Im\bigg(\int_0^z \omega\bigg)\geq 0.\]
Hence
\[f_\lambda(z)=\sqrt{-1}\lambda+\int_0^z\, \omega,\quad\lambda\in (\lambda_0,\,\infty),\]
is a family of projective functions on ${\Bbb D}\backslash\{z_j\}_{j=1}^\infty$ taking values in ${\Bbb H}$ and having monodromy in $\{w\mapsto w+t:\, t\in {\Bbb R}\}\subset{\rm PSL}(2,\,{\Bbb R})$. We claim that {\it $ds^2_\lambda:=f_\lambda^*\big(\widetilde{g_{\rm st}}\big)$ is a family of hyperbolic metrics representing the divisor $\fD=(h)$}. Actually,  $f_\lambda(z)$ equals $\big(-\sqrt{-1}\,a_j\big)\log\,(z-z_j)$ plus a a multi-valued holomorphic function near $z_j$, so $\{f,\,z\}=\frac{1}{(z-z_j)^2}+\cdots$ there. Hence $z_j$ is a cusp singularity of $ds^2_\lambda$.
Near each zero $w$ of $h(z)$ with multiplicity $\ell$, we have
$\frac{d}{dz}\big(f_\lambda(z)\big)=\big(-\sqrt{-1}\big)h(z)=(z-w)^\ell\, g(z)$, where $g(z)$ is holomorphic at $w$ and $g(w)\not=0$. Hence near $w$, $\{f,\,z\}=\frac{1-(\ell+1)^2}{(z-w)^2}+\cdots$, which implies
that $w$ is a cone singularity of $ds^2_\lambda$ with angle $2\pi(1+\ell)$.

By now, we have proved the statements in the first paragraph of Example \ref{exam:disc}. For the second one, we need the following elementary lemma.

\begin{lemma}
\label{lem:zeros}
Let $a_1,\cdots, a_N$ be $N\geq 2$ positive numbers and $z_1,\cdots, z_N$ distinct complex numbers in the disc
$\{|z|<R\}$, where $R$ is a positive constant. Then, the rational function $\sum_{j=1}^N\,\frac{a_j}{z-z_j}$ has $(N-1)$ zeros {\rm (}counting multiplicities{\rm )} on the disc $\{|z|<R\}$. The meromorphic function $h$ on ${\Bbb C}\backslash \{|z|=1\}$ in Example \ref{exam:disc} has no zero in $\{z\in {\Bbb C}:\, |z|\geq 1\}$.
\end{lemma}

\begin{proof} Taking a complex number $\xi$ such that $\Im \xi\leq -R$, we find
\[\Im\,(\xi-z_j)<0\quad {\rm and}\quad \Im\,\frac{a_j}{\xi-z_j}>0,\]
which implies $\sum_{j=1}^N\,\frac{a_j}{\xi-z_j}\not=0$. Observing that $\Im\, z=-R$ is a tangent line to
 the circle $\{|z|=R\}$, we could prove the first statement by arguing on each half plane defined by each tangent line to the circle and disjoint from the disc $\{|z|<R\}$. The second one follows from the similar argument. We also note that $h$ extends holomorphically to each point on the circle $\partial{\Bbb D}=\{|z|=1\}$, which is not a limit point of $\{z_j\}$.
\end{proof}

\nd Then we prove that {\it $h_0(z)=:\sum_{j=1}^\infty\, \frac{a_j}{z-z_j}$ has the same number of zeros as
$f_N(z):=\sum_{|z_j|\leq r_N}\, \frac{a_j}{z-z_j}$
on the disc $\{|z|<r_N:=1-\frac{1}{2N}\}$ when $N$ is sufficiently large, where $a_j=\frac{1}{2j^3(2j+1)}$ and $z_j=1-\frac{1}{2j-1}$}. At first we show that
on the circle $\{|z|=r_N\}$ there holds $|f_N(z)|>|g_N(z)|$ when $N$ is sufficiently large, where
\[ g_N(z):=h_0(z)-f_N(z)=\sum_{|z_j|>r_N}\,\frac{a_j}{z-z_j}.\]
In fact, on the circle $\{|z|=r_N\}$, we have
\[|g_N(z)|\leq \sum_{|z_j|>r_N}\,\frac{a_j}{|z_j|-r_N}\leq 2N(2N+1) \sum_{|z_j|>r_N}\,a_j\leq \sum_{j>N}\,\frac{1}{j^2}.\]
Moreover, denoting by $z= i r_N e^{i\theta}$ a point $z$ on this circle, by computation, we have
\[f_N(z)=e^{-i\theta}\,\sum_{|z_j|<r_N}\, \frac{a_j}{i r_N-z_j e^{-i\theta}}
=e^{-i\theta}\,\sum_{|z_j|<r_N}\, \frac{a_j\big(-i r_N-z_j e^{i\theta}\big)}{|i r_N-z_j e^{-i\theta}|^2},\]
and recalling $z_1=0$, we obtain
\[|f_N(z)|=|e^{i\theta}f_N(z)|\geq \left|\Im\, \big(e^{i\theta}f_N(z)\big)\right|
=\sum_{|z_j|<r_N}\,\frac{a_j(r_N+z_j \sin\,\theta)}{|i r_N-z_j e^{-i\theta}|^2}\geq \frac{a_1}{r_N}\geq a_1 \]
and prove the inequality $|f_N(z)|>|g_N(z)|$ on the circle.
Since both $h_0$ and $f_N$ have $N$ simple poles on the disc $\{|z|<r_N\}$ and
$|h_0(z)-g_N(z)|=|f_N(z)|>|-g_N(z)|$ on the circle $\{|z|=r_N\}$, by the Rouch\' e theorem,
$h_0(z)$ has the same number of zeros as $f_N(z)$ on the disc  $\{|z|<r_N\}$. By Lemma \ref{lem:zeros}, counting multiplicities, we find that $f_N(z)$ has $(N-1)$ zeros on the disc. Therefore, $h_0(z)$ has infinitely many zeros on ${\Bbb D}$. By now we have completed the exposition of Example \ref{exam:disc}. \\

At last, we propose the following two questions.\\


{\bf Question 3.1.} Use the notions in Theorem 1.1 and assume that $\theta_j$'s are non-negative rational numbers. {\it What is the necessary and sufficient condition for $D=\sum_{j=1}^n\, (\theta_j-1)P_j$ under which the monodromy group of the developing map $f$ of the hyperbolic metric $ds^2$ representing $D$ on a compact Riemann surface $X$ is discrete in ${\rm PSU(1,\,1)}$?} It is the case when
$\theta_j\in \{0,1/2, 1/3,\cdots\}$ by the Uniformization Theory. Also a conceptual necessary and sufficient condition was given in \cite[Theorem 3.29]{Beuk2007} for the case of 3 singularities on the Riemann sphere, which has yet to be expressed in terms of $\theta_1,\, \theta_2$ and $\theta_3$. \\

{\bf Question 3.2.}  It is interesting for us to investigate the existence and the uniqueness of hyperbolic metrics with isolated singularities on noncompact Riemann surfaces, which seems to be an open problem to the best of our knowledge. We will give a partial answer to this problem in \cite{Feng2019}.\\

{\bf Question 3.3.} Does $h(z)=\sum_{j=1}^\infty\, \frac{a_j}{z-z_j}$ in Example \ref{exam:disc} always have infinitely many zeros on ${\Bbb D}$? Note that $h$ never vanishes outside ${\Bbb D}$. 

\section{Fundings}

The first author is supported in part by China Scholarship Council, the third author by
 ERC ALKAGE  and the last author by the National Natural Science Foundation of China (grant no. 11571330) and
the Fundamental Research Funds for the Central
Universities.


\end{document}